\documentclass{article}
\usepackage{times}
\usepackage[hyperindex=true,pageanchor=true,hyperfigures=true,backref=false]{hyperref} 
\usepackage{amsmath}
\usepackage{amssymb}
\usepackage{amsthm}
\usepackage[capitalise,nameinlink]{cleveref}
\usepackage{mathtools}
\usepackage{tikz-cd}

\usepackage{enumitem}

\usepackage{url}
\usepackage{dsfont} 
\usepackage[toc]{appendix}
\DeclareSymbolFontAlphabet{\Bbb}{AMSb}
\usepackage[T1]{fontenc}

\usepackage{mleftright,xparse}

\usepackage{thm-restate,amsthm}

\newtheorem*{assumptionV}{Assumption V}  

\newlength{\myleftmargin}
\usepackage{amsmath}
\usepackage{amssymb}
\usepackage{amsthm}
\usepackage{color}
\DeclareSymbolFontAlphabet{\Bbb}{AMSb}

\newtheorem{theorem}{Theorem}[section]
\newtheorem{lemma}[theorem]{Lemma}

\newtheorem{corollary}[theorem]{Corollary}

\newtheorem{remark}[theorem]{Remark}

\newcommand{\Xbn}{X^{[n]}}
\newcommand{\Xsbn}{X^{(n)}}

\newlength{\fixboxwidth}
\newcommand{\mycdot}{\,\cdot\,}

\title{Convergence Rates for Realizations of Gaussian Random Variables}

\author{Daniel Winkle, Ingo Steinwart, and Bernard Haasdonk\\
University of Stuttgart\\
\small Faculty 8: Mathematics and Physics\\
\small Stuttgart, Germany \\
\texttt{\small daniel.winkle@mathematik.uni-stuttgart.de}\\
\texttt{\small ingo.steinwart@mathematik.uni-stuttgart.de}\\
\texttt{\small haasdonk@mathematik.uni-stuttgart.de}
}

\begin{document}

\maketitle

\begin{abstract}
This paper investigates the approximation of Gaussian random variables in Banach spaces, focusing on the high-probability bounds for the approximation of Gaussian random variables using finitely many observations. 
We derive non-asymptotic error bounds for the approximation of a Gaussian process $ X $ by its conditional expectation, given finitely many linear functionals. 
Specifically, we quantify the difference between the covariance of $ X $ and its finite-dimensional approximation, establishing a direct relationship between the quality of the covariance approximation and the convergence of the process in the Banach space norm. 
Our approach avoids the reliance on spectral methods or eigenfunction expansions commonly used in Hilbert space settings, and instead uses finite, linear observations. This makes our result particularly suitable for practical applications in nonparametric statistics, machine learning, and Bayesian inference. 
\end{abstract}



\section{Introduction}\label{sec:intro}
Interpolation or kriging with Gaussian random variables is a widely used tool for analyzing functional data, with applications in computer experiments \cite{computerexperiments}, machine learning \cite{RasmussenWilliams2006}, and geostatistics \cite{geostatistic}. These methods naturally encode prior uncertainty and allow for coherent probabilistic inference, making them central to Bayesian modeling frameworks \cite{bayes}. In particular, Gaussian process models have become essential in scientific computing and statistical learning, where they are used to approximate expensive computer models \cite{KennedyOHagan}, reconstruct spatial fields \cite{spatial}, and quantify uncertainty in inverse problems \cite{Stuart2010}.

A key reason for the widespread success of Gaussian models in these domains is their analytic tractability. Gaussian distributions are fully characterized by their mean and covariance structures, even in infinite-dimensional settings such as function spaces \cite{linearconstraints}. This property enables closed-form conditioning, which is crucial for both theoretical analysis and practical computation. Additionally, many common function spaces, such as Sobolev and Hölder spaces, can be naturally related to the support or concentration sets of Gaussian measures \cite{Kuo, Lifshits}, offering a flexible and interpretable framework for regularization and prior design.

Building on these foundations, the theory of Gaussian interpolation is particularly well developed, especially with regard to posterior variances.
In the noiseless setting, kriging amounts to interpolating a function using samples drawn from a Gaussian prior, yielding a posterior process that matches observations exactly while minimizing uncertainty elsewhere. Under standard assumptions on the sampling geometry and the choice of covariance kernel, most notably for Matérn and squared exponential kernels, convergence rates with known constants for the posterior variance have been established \cite{pgreedy,daniel} and \cite[Chapter 11]{wendland}. These results rely on connections between Gaussian processes and reproducing kernel Hilbert spaces (RKHSs), and they form the theoretical foundation for adaptive sampling strategies and design of experiments. The convergence of posterior variance is particularly relevant in applications where function-valued quantities are observed only at discrete locations, such as in nonparametric regression \cite{Wahba901}, surrogate modeling \cite{RasmussenWilliams2006}, and high-dimensional inference \cite{Stuart2010}, since the posterior variance describes the uncertainity of the predictions at unobserved locations.

In this paper, we show that convergence rates of the posterior variance imply concentration of the posterior mean.
More precisely, we derive a concentration inequality that bounds the posterior mean of a Gaussian random variable with the help of the posterior variance.

In the literature, many results can be found that focus on conditioning on point evaluations, see e.g. \cite{pointregression}. In contrast, we also consider general linear measurement operators as in \cite{disintegration}. One of those linear measurement operators can be given by orthogonal projections onto eigenfunctions of integral operators. This perspective allows us to extend recent results from \cite{GPonSphere}, where we obtain the slightly sharper convergence rates as in \cite[Theorem 5.1]{GPonSphere}, but instead of almost sure convergence, we obtain a concentration inequality.

Moreover, \cite{RuiTao} established a concentration inequality for kriging, that is, for point evaluation, with misspecified stationary Gaussian processes. While their results yield faster convergence rates in the case of polynomial decay, we demonstrate that in the case of exponential convergence, our results can lead to improved rates. Further details are provided in Section \ref{sec:examples}.

Beyond projections, our framework accommodates broader classes of measurement operators, including those defined via bounded linear mappings such as partial differential operators. This generality is particularly relevant in the context of physics-informed Gaussian processes \cite{physicskrigging}, where the authors used a Monte Carlo approach to obtain approximations in various functional norms. In that setting, particular emphasis is placed on $L^2$, $L^\infty$, and Sobolev norms, which are natural when modeling physical quantities governed by differential constraints.

In comparison to our setting, much of the existing literature focuses on noisy observations. In particular, several influential results have been developed in the context of posterior contraction \cite{banachcontraction,contractionrates} and learning rates \cite{Simon}, where convergence rates typically scale no faster than $\sqrt{n}$. This limitation is largely due to the presence of observational noise. By restricting our attention to the noiseless case, we are able to obtain significantly faster rates, at the cost of narrower applicability. Nevertheless, our results complement these existing frameworks by highlighting the benefits of posterior variance decay in high-accuracy, zero-noise regimes.

Achieving these faster convergence rates, however, introduces its own technical challenges. A central difficulty is that realizations of infinite dimensional Gaussian random variables almost surely lie outside their associated RKHS \cite{driscol,parzen}. This well-known phenomenon precludes the direct application of classical kernel approximation results \cite{wendland} and necessitates addressing the so-called escaping-the-native-space problem \cite{escapingnativespace}. 
When these challenges are properly addressed, one can obtain asymptotic convergence rates faster than $\sqrt{n}$. The precise rate depends on both the smoothness of the underlying Sobolev space and the dimensionality of the domain, see \cite{maximumlikelihood}. 
The question is thus answered for Sobolev spaces, but for general RKHSs this is still an open problem.

To overcome this, we utilize a new class of scalable RKHSs introduced in \cite{samplespace}, and prove an escaping the native space estimate (Lemma~\ref{lem:inequality mainMehr}). This is combined with the Karhunen–Loève expansion and a known concentration inequality for $\chi^2$-random variables, yielding our main result.

Finally, we emphasize that our framework is expressed in terms of Gaussian random variables rather than Gaussian processes, following the distinction discussed in \cite{GRVandGP}. This choice allows for a more general formulation.

The remainder of the paper is organized as follows. In Section \ref{sec:prelim}, we introduce the basic notation and background material. Section \ref{sec:results} presents our main results, including the concentration inequality and its implications. In Section \ref{sec:examples}, we illustrate the applicability of our framework through representative examples. The proofs are collected in Section \ref{sec:proofs}. Lastly, all the technical details that do not provide further insight can be found in Appendix \ref{sec:appendix}.

\section{Preliminaries}\label{sec:prelim}\label{sec:2}

Throughout this work, $(\Omega, \mathcal{A}, \mu)$ denotes a probability space, $ E $ a separable Banach space, and $ E' $ the dual of $E$. Moreover we denote by $ B_E $ the closed unit ball in $ E $. 

\subsection*{Bochner Spaces}

Given $ 1 \leq p < \infty $ and a Banach space $ E $, the Bochner space $ L^p:=L^p(\mu,E) $ consists of all (equivalence classes of) strongly measurable functions $ X: \Omega \to E $ such that
\begin{align*}
\Vert X \Vert_{L^p(\mathcal{A},E) }^p :=\int_\Omega \Vert X(\omega) \Vert_E^p \, \textup{d} \mu(\omega) < \infty.
\end{align*}
Strong measurability here means that $ X $ is the almost everywhere limit of simple functions taking values in $ E $. For more details, we refer to \cite[Chapter 1]{martingal}. We also define for $X \in L^1(\mu,E)$ the expectation as 
\begin{align*}
\mathbb{E}X:= \int_\Omega X \, \textup{d} \mu.
\end{align*}

\subsection*{Conditional Expectation}

Let $ \mathcal{A}_0 \subset \mathcal{A} $ be a sub-$ \sigma $-algebra and let $ X \in L^1(\mathcal{A},E) $. A function $ \mathbb{E}(X|\mathcal{A}_0) \in L^1(\mathcal{A}_0,E) $ is called a conditional expectation of $ X $ with respect to $ \mathcal{A}_0 $ if
\begin{align*}
\int_A \mathbb{E}(X|\mathcal{A}_0) \, \textup{d} \mu = \int_A X \, \textup{d} \mu, \quad \text{for all } A \in \mathcal{A}_0.
\end{align*}
The existence and almost sure uniqueness (up to $ \mu $-null sets) of conditional expectations in the Banach space setting is guaranteed by \cite[Theorem 2.6.18 and Theorem 2.6.20]{martingal}.

\subsection*{Gaussian Random Variables}

A random variable $ X: \Omega \to \mathbb{R} $ is called a (one-dimensional) Gaussian random variable if there exist $ \mu_X, \sigma_X \in \mathbb{R} $ such that
\begin{align*}
\mathbb{E}\left( e^{i t X} \right) = e^{i t \mu_X - \frac{1}{2} \sigma_X^2 t^2}, \quad \text{for all } t \in \mathbb{R}.
\end{align*}
In this case, we write $ X \sim \mathcal{N}(\mu_X,\sigma_X^2) $. We note that $ \sigma_X = 0 $ yields a Dirac measure at $ \mu_X $, while $ \sigma_X > 0 $ corresponds to the usual normal distribution.

A random variable $ X: \Omega \to E $ is called Gaussian if for every $ e' \in E' $, the real-valued random variable $ e'(X) :\Omega \to \mathbb{R} $ is Gaussian. We note that separability of $E$ ensures measurability of $X$. Additionally, Fernique’s theorem ensures $X \in L^p(\mu,E) \, \text{for all } p \geq 1$, see \cite[Theorem 5.3]{Fernique53}. We call $X$ centered if $\mathbb{E}(X)=0$.

Given a sequence \( (e_j') \subset E' \), we define the conditional expectation
\begin{align}\label{eq:ConditionalRandomVariable}
\Xsbn := \mathbb{E}\left( X \mid \sigma\left( e_1'(X), \cdots , e_n'(X) \right) \right)
\end{align}
It follows from \cite[Theorem 3.10.1]{Bogachev1998} that $ \Xsbn $ is again a Gaussian random variable.

Let $ X: \Omega \to E $ be a Gaussian random variable. The \textit{covariance operator} $ \textup{cov}(X): E' \to E$ is defined by
\begin{align*}
\textup{cov}(X)(e') := \int_\Omega e'(X(\omega)-\mathbb{E}X) \cdot (X(\omega)-\mathbb{E}X) \, \textup{d} \mu(\omega), \quad \text{for all } e' \in E'.
\end{align*}

Let $ T $ be a compact metrizable space equipped with the Borel-algebra and a $ \sigma $-finite measure \( \lambda \). 
In the case where $X$ takes values in $C(T)$ the space of continuous functions on $T$, we associate a kernel function $k_X : T \times T \to \mathbb{R}$ to $X$. This kernel function is defined by
\begin{align*}
k_X(t, s) := \int_\Omega \left(X(t, \omega) - \mathbb{E}X(t, \mycdot)\right) \cdot \left(X(s, \omega) - \mathbb{E}X(s, \mycdot)\right) \, \mathrm{d}\mu(\omega).
\end{align*}
Moreover, we have
\begin{align}\label{eq:StetigeNorm}
\Vert \textup{cov}(X) - \textup{cov}(\Xsbn) \Vert_{C(T)' \to C(T)}= \sup_{t \in T} | k_X(t,t)-k_{\Xsbn}(t,t) | ,
\end{align} 
as shown in Lemma \ref{lem:stetigeNorm}. 
\section{Main Results}\label{sec:results}
In this section, we present the main results of our work. 
We begin with a general theorem that also covers the case where, at each step, the condition in \eqref{eq:ConditionalRandomVariable} is not only increased by a single additional $e_j' \in E'$ but rather by $d_j \in \mathbb{N}$ finitely many functionals $(e_{j,m}')_{m=1}^{d_j} \subset E'$. In this case, we write 
\begin{align}\label{eq:ConditioningMehr}
\Xbn:= \mathbb{E} \left (X \, | \,\sigma \left( \left( \left(e_{j,m}'(X) \right)_{m=1}^{d_j} \right)_{j=1}^n \right) \right).
\end{align}
Our main result shows that convergence rates for the posterior variance also imply convergence rates for the realizations. We then derive our other results from this theorem and in particalur we treat the cases, where the posterior variance converges at a polynomial rate and at an exponential rate.
\begin{theorem}\label{Theorem:MehrMain Result}
Let $X$ be a Gaussian random variable and $(e_{j,m}')_{m=1}^{d_j} \subset E'$. Suppose there exists a monotone decreasing sequence $(c_j) \in \ell^1(\mathbb{N})$ such that
\begin{align*}
\Vert \textup{cov}(X)-\textup{cov}(\Xbn) \Vert_{E'\to E} \leq c_n
\end{align*}
for all $n \in \mathbb{N}$.
Then, for every monotone increasing sequence $(a_j)$ satisfying $(d_j/a_j) \in \ell^1(\mathbb{N})$, $(a_j (c_{j-1}-c_j)) \in \ell^1(\mathbb{N})$, $a_j c_j \to 0$, and for all $n \in \mathbb{N}$, $\tau>0$ we have 
\begin{align*}
\mu\left( \Vert X-\Xbn \Vert_E \leq \sqrt{ 5 \max\{1,\tau\}  \left(\sum_{j=n+1}^\infty  a_j (c_{j-1}-c_j) \right) \left( \sum_{j=n+1}^\infty \frac{d_j}{a_j} \right)} \right) > 1-\mathrm{e}^{-\tau}.
\end{align*} 
\end{theorem}

In the case where the posterior variance converges at a polynomial rate, we obtain the following corollary.

\begin{corollary}\label{Corollary:PolynomHigher}
Let $X$ be a Gaussian random variable and let $(e_{j,m}') \subseteq E'$. Suppose there exist constants $C, C_d > 0$, $\alpha > 1+\beta \geq 1$, such that
\begin{align*}
\Vert \textup{cov}(X)-\textup{cov}(\Xbn) \Vert_{E'\to E} \leq C (n+1)^{-\alpha}, \qquad \textup{and} \qquad d_n \leq C_d(n+1)^{\beta}
\end{align*}
for all $n\in \mathbb{N}$. Then, for $n \geq 1$ and $\tau>0$, we have
\begin{align*}
\mu \left( \Vert X-\Xbn \Vert_{E} \leq \frac{\sqrt{20 \cdot 2^\beta C C_d \max\{1,\tau \}}}{\alpha-\beta-1} n^{\frac{-\alpha+\beta+1}{2}} \right) > 1 -\mathrm{e}^{-\tau}.
\end{align*} 
\end{corollary}

The next theorem should be viewed as a simplified version of our main result, in which only one functional is added at each conditioning step.

\begin{theorem}\label{Theorem:simple Main Result}
Let $X$ be a Gaussian random variable and $(e_j') \subset E'$ a sequence of functionals for which there exists a positive, monotone decreasing sequence $(c_j) \in \ell^1(\mathbb{N})$ such that
\begin{align}\label{eq:covariance condition}
\Vert \textup{cov}(X)-\textup{cov}(\Xsbn) \Vert_{E'\to E} \leq c_n,
\end{align}
for all $n \in \mathbb{N}$
Assume further that $(c_{j-1}-c_j)$ is monotone decreasing, $\left( \sqrt{c_{j-1}-c_j} \right)\in \ell^1(\mathbb{N})$, and $c_j/\sqrt{c_{j-1}-c_j} \to 0$. Then, for all $n \in \mathbb{N}$ and $\tau>0$ we have 
\begin{align*}
\mu \left( \Vert X-\Xsbn \Vert_E \leq \sqrt{5 \max\{1,\tau\}} \sum_{j=n+1}^\infty \sqrt{c_{j-1}-c_j} \right) > 1 -\mathrm{e}^{-\tau}.
\end{align*} 
\end{theorem}

The following corollary specializes Corollary \ref{Corollary:PolynomHigher} to the case in which only one functional is added at each conditioning step.

\begin{corollary}\label{Corollary:Polynom}
Let $X$ be a Gaussian random variable and $(e_j') \subset E'$ be a sequence for which there exist $C>0$ and $\alpha>1$ such that
\begin{align*}
\Vert \textup{cov}(X)-\textup{cov}(\Xsbn) \Vert_{E'\to E} \leq C (n+1)^{-\alpha}
\end{align*}
for all $n\in \mathbb{N}$. Then, for $n \geq 1$ and $\tau>0$, we have
\begin{align*}
\mu \left( \Vert X-\Xsbn \Vert_E \leq \frac{ \sqrt{ 20  C  \max\{1,\tau \} } }{(\alpha-1)}    n^{\frac{1-\alpha}{2}}  \right) > 1 -\mathrm{e}^{-\tau}.
\end{align*} 
\end{corollary}

Similarly, if the posterior variance converges at an exponential rate, we obtain the following bound. 

\begin{corollary}\label{Corollary:exponent}
Let $X$ be a Gaussian random variable and $(e_j') \subset E'$ be a sequence for which there exist $C_1,C_2>0$ and $ \alpha > 1$ such that
\begin{align*}
\Vert \textup{cov}(X)-\textup{cov}(\Xsbn) \Vert_{E'\to E} \leq C_1 \mathrm{e}^{-C_2 n^{1/\alpha}}
\end{align*}
for all $n \in \mathbb{N}$. Then, for $n > \left( \frac{11}{C_2} (\alpha - 1) \right)^{\alpha}$ and $\tau>0$, we have
\begin{align*}
\mu \left( \Vert X-\Xsbn \Vert_E \leq \sqrt{ \frac{121 C_1 C_2 \alpha \max \{1,\tau\}}{20} }   \left(\frac{C_2}{2} \right)^{\alpha-2}  (n-1)^{\frac{\alpha-1}{2\alpha}}  \mathrm{e}^{-\frac{C_2}{2}(n-1)^{1/\alpha}} \right) > 1 -\mathrm{e}^{-\tau}.
\end{align*} 
For $\alpha=1$ we obtain for all $\tau>0$ and $n \in \mathbb{N}$
\begin{align*}
\mu \left( \Vert X-X_n \Vert_E \leq \sqrt{5 \max\{1,\tau \} C_1 \cdot (\mathrm{e}^{C_2} - 1) } \cdot \frac{\mathrm{e}^{C_2}}{\sqrt{\mathrm{e}^{C_2}}-1} \mathrm{e}^{-\frac{C_2}{2} n} 
\right) > 1 -\mathrm{e}^{-\tau}.
\end{align*} 
\end{corollary}
\begin{remark}
Assume that $E=C(T)$ and that the covariance kernel of $X$ can be written as $k_X(t,s)=\kappa(t-s)$ for all $t,s \in T$, where $\kappa: \mathbb{R}^d \to \mathbb{R}$. In addition assume that there exists an $0< \alpha \leq 1$ such that
\begin{align}\label{eq:RuiCondition}
\int_{\mathbb{R}^d} \Vert \omega \Vert^\alpha \hat{\kappa}(\omega) \, \textup{d} \mu < \infty,
\end{align} 
where $\hat{\kappa}$ denotes the Fourier transform of $\kappa$. 
If we consider conditioning with respect to point evaluations and assume that \eqref{eq:covariance condition} holds then the results in \cite{RuiTao} lead to the concentration inequality 
\begin{align*}
\mu \left( \Vert X-X_n \Vert_{C(T)} \leq C \left(\sqrt{c_n}\ln\left( \frac{1}{c_n} \right) + \sqrt{\tau c_n} \right)  \right) > 1 -\mathrm{e}^{-\tau},
\end{align*}
for all $n \in \mathbb{N}$ and all $\tau>0$ where $C>0$ is a suitable constant. This is comparable to Theorem \ref{Theorem:simple Main Result}. We will discuss in Section \ref{sec:examples}, when which result is better.
\end{remark}

\section{Examples}\label{sec:examples}\label{sec:4}

We focus our examples on Gaussian random variables taking their values in the Banach space $E=C([0,1]^d)$. For simplicity, we assume that $ \mathbb{E}(X) = 0 $ and describe the Gaussian random variable $X$ using the covariance operator defined on point evaluations $ \delta_t(X)(\omega) := X_t(\omega):=X(t,\omega) $ for $ \omega \in \Omega$ and $ t \in [0,1]^d$.

\subsection{Sobolev Covariance}
Let $ X: \Omega \to C([0,1]^d) $ be a centered Gaussian random variable with kernel function given by
\begin{align*} 
k_X(t,s) = \frac{2^{1-(s-d/2)}}{\Gamma(s-d/2)} \| r - t \|^{s-d/2} K_{s-d/2} (\| r - t \|),
\end{align*}
where $s>d$, $ K_{s-d/2} $ is a modified Bessel function of the second kind, and $ \Gamma $ is the Gamma function. This kernel is known as the Matérn kernel, whose RKHS is a Sobolev space $ H^s([0,1]^d) $, see \cite{Solin2020}. 

Assume that we have chosen points $ (t_n) \in [0,1]^d $, such that for some constant $ C > 0 $, we have
\begin{align*}
\sup_{t \in [0,1]^d} | k_X(t,t) - k_{\Xsbn}(t,t) | \leq C (n+1)^{-\frac{2s}{d}+1}, \quad \textup{for all} \, \, n \in \mathbb{N}.
\end{align*}
Such points can be constructed by algorithms like $P$-greedy, see \cite[Corollary 2.2]{pgreedy}. Applying \eqref{eq:StetigeNorm} and Corollary \ref{Corollary:Polynom} leads to
\begin{align*}
\mu\left( \Vert X - \Xsbn \Vert_{C([0,1]^d)} \leq \frac{\sqrt{20 C \max\{1,\tau\}}}{\frac{2s}{d}-2} n ^{-\frac{s}{d}+1} \right) > 1- \mathrm{e}^{-\tau}
\end{align*} 
for $n\geq 1$ and $\tau>0$. 
We note that this converges to zero only if \( s > d \), recall that this condition also ensures that almost all realizations of $X$ are contained in the Sobolev space $H^r([0,1]^d)$ for all $r \in (d/2,s-d/2)$, see \cite[Example 5.6]{IngoGPpath}. 

We note that the convergence rates provided in \cite{RuiTao} lead to a convergence rate of 
\begin{align}\label{eq:sobolevfaster}
\mu \left( \Vert X - \Xsbn \Vert_{C([0,1]^d)} \leq C \left( n^{-\frac{s}{d}+\frac{1}{2}} \sqrt{\ln(n)} +\sqrt{\tau} n^{-\frac{s}{d}+\frac{1}{2}} \right) \right)> 1 - \mathrm{e}^{-\tau} ,
\end{align}
for some $C>0$. The rate in \eqref{eq:sobolevfaster} is faster than the one provided in our setting, however it only holds for stationary processes satisfying \eqref{eq:RuiCondition}.
\color{black}
\subsection{Gaussian Covariance}
Let $ X: \Omega \to C([0,1]^d)$ be a centered Gaussian random variable with the kernel function 
\begin{align*}
k_X(t,s) = \mathrm{e}^{-\| t - s \|_2^2}, \quad t,s \in [0,1]^d.
\end{align*}
We identify points $t_n$ with the functionals $ \delta_{t_n} $, meaning that the conditional expectation $ \Xsbn $ is given by 
\begin{align*}
\Xsbn = \mathbb{E}(X \mid \sigma((X_{t_j})_{j=1}^n)).
\end{align*}
Assume that the points $ (t_n) \in [0,1]^d $ are chosen such that there exist constants $C_1,C_2>0$ with
\begin{align*}
\sup_{t \in [0,1]^d} | k_X(t,t) - k_{\Xsbn}(t,t) | \leq C_1 \mathrm{e}^{-C_2 n^{1/d}}, \quad \textup{for all} \, \, n \in \mathbb{N}.
\end{align*}
Such points do exist and can be constructed, for instance using $P$-greedy algorithms, see \cite[Corollary 2.2]{pgreedy}. We then apply \eqref{eq:StetigeNorm} and Corollary \ref{Corollary:exponent} to obtain 
\begin{align*}
\mu\left( \Vert X - \Xsbn \Vert_{C([0,1]^d)} \leq \sqrt{ \frac{ 121 C_1 C_2 d \max\{1,\tau\} }{20}} \left(\frac{C_2}{2} \right)^{d-2} (n-1)^{\frac{d-1}{2d}} \mathrm{e}^{-\frac{C_2}{2}(n-1)^{1/d}} \right) > 1 -\mathrm{e}^{-\tau},
\end{align*}
for $d>1$, $ n > \left( \frac{11}{C_2} (d-1) \right)^d + 1 $, and $\tau>0$.

For $d=1$, $ n \in \mathbb{N}$, and $\tau>0$ we obtain 
\begin{align*}
\mu \left( \Vert X-X_n \Vert_E \leq \sqrt{5 \max\{1,\tau \} C_1 \cdot (\mathrm{e}^{C_2} - 1) } \cdot \frac{\mathrm{e}^{C_2}}{\sqrt{\mathrm{e}^{C_2}}-1} \mathrm{e}^{-\frac{C_2}{2} n} 
\right) > 1 -\mathrm{e}^{-\tau}.
\end{align*}

We note that the convergence rates provided in \cite{RuiTao} lead to a convergence rate of 
\begin{align}\label{eq:gausfaster}
\mu \left( \Vert X - \Xsbn \Vert_{C([0,1]^d)} \leq C \left( \mathrm{e}^{-\frac{C_2}{2} n^{1/d}} \cdot n^{\frac{1}{2d}} + \sqrt{\tau} \mathrm{e}^{-\frac{C_2}{2} n^{1/d }} \right) \right)> 1 - \mathrm{e}^{-\tau} 
\end{align}
for some $C>0$. The rate in \eqref{eq:gausfaster} is faster if $d\geq 3$ and for $d=2$ the rates coincide. For the case $d=1$ our setting provides an improvement in the realistic case $\tau \leq n^{1/d}$. Note, however that these improvements are only obtained for stationary processes satisfying \eqref{eq:RuiCondition}.
\color{black}

\subsection{Conditioning on Eigenfunctions}

In the case of $E=L^2(T)$, it is often more natural to condition on the eigenfunctions of the covariance operator rather than on point evaluations. Recall that for a separable Banach space, the covariance operator is nuclear, symmetric and positive, see \cite{CovarianceSymmetric}.  Let $(\lambda_j)$ denote the monotone decreasing eigenvalues and $(e_j)$ denote the eigenfunctions. By conditioning on $e_j':=\langle e_j, \mycdot \rangle_{L^2(\lambda)} \in L^2(\lambda)'$, we have by Lemma \ref{lem:Eigenvalue matches Norm}
\begin{align*}
\Vert \textup{cov}(X)-\textup{cov}(\Xsbn) \Vert_{L^2(T) \to L^2(T)} = \lambda_{n+1}.
\end{align*}
If we further assume that there exist constants $C>0$ and $\alpha>1$ such that $\lambda_{n+1} \leq C (n+1)^{-\alpha}$, as is the case for Sobolev spaces, then Corollary \ref{Corollary:Polynom} yields
\begin{align*}
\mu \left( \Vert X-\Xsbn \Vert_E \leq  \frac{ \sqrt{20 C \max\{1,\tau\}}}{\alpha-1}   n^{(1-\alpha)/2}  \right) > 1 -\mathrm{e}^{-\tau}.
\end{align*}

\subsection{Gaussian random variables on the product of spheres}
%
We now make the previous eigenfunction-based example more concrete by considering a product of spheres $T = \mathbb{S}^{\textbf{d}_1} \times \mathbb{S}^{\textbf{d}_2}$, following the setup in \cite{GPonSphere}. We note, $\mathbb{S}^{d} \subset \mathbb{R}^{d+1}$ denotes the sphere and $d,\textbf{d}_1,\textbf{d}_2 \in \mathbb{N}$.  To avoid repeating the full harmonic analysis framework, we refer the reader to \cite{GPonSphere}, but recall the key elements needed for our discussion.

Let $\Delta$ denote the Laplace-Beltrami operator on $\mathbb{S}^d$ with eigenvalues $\mu_j$ and corresponding eigenspaces $H_j(d)$. The dimension of $H_j(d)$ is denoted by $D_j(d)$, with $D_j(1)=2$ and, for $d \geq 2$,
\begin{align*}
D_j(d)= \textup{dim}(H_j(d))=(2j+d-1)\frac{j+d-2}{j!(d-1)!} \leq c_d (j+1)^{d-1}
\end{align*} 
for some $c_d\geq 1$, see \cite[Equation (5.2)]{GPonSphere}. We denote by $S_{j,l}^{d}$ the classical spherical harmonics on $\mathbb{S}^d$, see \cite{sphericalharmonic}, which form an orthonormal basis (ONB) of $H_j(d)$. 

For $\textbf{j}=(j_1,j_2) \in \mathbb{N}^2$ and $\textbf{m} =(m_1,m_2) \in \mathbb{N}^2$, let $r_{\textbf{j},\textbf{m}} \sim \mathcal{N}(0,1)$ be i.i.d., and let $B_{\textbf{j}}>0$ satisfy
\begin{align}\label{eq:EigenvaluedecaySphere}
\sum_{\textbf{j} \in \mathbb{N}_0^2} B_{\textbf{j}} D_{j_1}(\textbf{d}_1) D_{j_2}(\textbf{d}_2) < \infty.
\end{align}
 We define the Gaussian random variable 
\begin{align*}
X:= \sum_{\textbf{j} \in \mathbb{N}_0^2} \sqrt{B_{\textbf{j}}} \sum_{m_1=1}^{D_{k_1}(\textbf{d}_1)} \sum_{m_2=1}^{D_{k_2}(\textbf{d}_2)}  r_{\textbf{j},\textbf{m}} S_{j_1,m_1}^{\textbf{d}_1} \cdot S_{j_2,m_2}^{\textbf{d}_2}.
\end{align*}
By construction, the covariance operator of $X$ has eigenvectors $S_{j_1,m_1}^{\textbf{d}_1} , S_{j_2,m_2}^{\textbf{d}_2}$ with eigenvalues $B_{\mathbf{j}}$, so that
\begin{align*}
\Vert \textup{cov}(X) \Vert_{L^2(\lambda) \to L^2(\lambda)} = \max_{\mathbf{j}} |B_{\mathbf{j}}|.
\end{align*}
By conditioing $X$ on $\left( S_{j_1,m_1}^{\textbf{d}_1} \cdot S_{j_2,m_2}^{\textbf{d}_2} \right)_{\textbf{m}=(1,1)}^{\left( D_{j,1}(\textbf{d}_1),D_{j_2}(\textbf{d}_2) \right)}$ for $|\textbf{j}| \leq n$, we have $d_n=\sum_{l=0}^n D_{n-l}(\textbf{d}_1) D_l(\textbf{d}_2)$. Furthermore, by Lemma \ref{lem:Eigenvalue matches Norm} we obtain
\begin{align*}
\Xbn = \sum_{| \textbf{j} | \leq n} \sqrt{ B_{\textbf{j}}} \sum_{m_1=1}^{D_{j_1}(\textbf{d}_1)} \sum_{m_2=1}^{D_{j_2}(\textbf{d}_2)}  r_{\textbf{j},\textbf{m}} S_{j_1,m_1}^{\textbf{d}_1} \cdot S_{j_2,m_2}^{\textbf{d}_2}.
\end{align*}
For $\alpha,C>0$, let the coefficients satisfy
\begin{align}\label{eq:eigenwertdecay}
B_\textbf{j} \leq C (1+ | \textbf{j} | )^{-2\alpha-\textbf{d}_1-\textbf{d}_2},
\end{align}
where we note that \eqref{eq:eigenwertdecay} implies \eqref{eq:EigenvaluedecaySphere}.
By Lemma \ref{lem:Eigenvalue matches Norm} we have
\begin{align*}
\Vert \textup{cov}(X)-\textup{cov}(\Xbn) \Vert_{L^2(\lambda)\to L^2(\lambda)} \leq C (1+ n )^{-2\alpha-\textbf{d}_1-\textbf{d}_2}.
\end{align*}
Counting the functionals one adds in each step we estimate $d_n$ by
\begin{align*}
d_n=\sum_{l=0}^n D_{n-l}(\textbf{d}_1) D_{l}(\textbf{d}_2) \leq \sum_{l=0}^n c_{\textbf{d}_1} c_{\textbf{d}_2} (n-l+1)^{\textbf{d}_1-1} (l+1)^{\textbf{d}_2-1} \leq c_{\textbf{d}_1} c_{\textbf{d}_2} (1+n)^{\textbf{d}_1+\textbf{d}_2-1}.
\end{align*}
Applying Corollary \ref{Corollary:PolynomHigher}, we obtain the concentration inequality 
\begin{align*}
\mu \left(  \Vert X-\Xbn \Vert_{L^2(\lambda)} \leq \frac{1}{2 \alpha} \sqrt{20 \cdot 2^{\textbf{d}_1+\textbf{d}_2-1} Cc_{\textbf{d}_1} c_{\textbf{d}_2} \max\{1,\tau \}} n^{-\alpha}  \right) > 1-\mathrm{e}^{-\tau}.
\end{align*}
In \cite[Theorem 5.1]{GPonSphere} the authors have shown an asymptotic almost surely bound $ \Vert X-\Xbn \Vert_{L^2(\lambda)} \leq n^{-\gamma}$ for all $0<\gamma < \alpha$. In contrast, we obtain the 'limiting' $n^{-\alpha}$ as convergence rate but instead of holding almost surely, our bound is formulated in terms of a concentration inequality with explicit constants.

\section{Proofs of main results}\label{sec:proofs}\label{sec:5}
The proof of Theorem \ref{Theorem:MehrMain Result} is organized in three main steps.
First, we introduce the \emph{reproducing covariance space} $H_X$ and demonstrate that the convergence rates derived in our results also hold for scalable RKHS $H_{a,V,d}$. For a detailed treatment of scalable RKHS, we refer the reader to \cite{samplespace}.
Next, we establish a crucial concentration inequality and show that the Gaussian random variable $X$ belongs to $H_{a,V,d}$ almost surely.
Finally, we combine these ingredients to prove Theorem \ref{Theorem:MehrMain Result}. Once this is established, Theorem \ref{Theorem:simple Main Result} and all corollaries follow directly.

\subsection*{Reproducing Covariance Space}
Given a Gaussian random variable $X$, we define the reproducing covariance space $H_X \subset E''$ as the RKHS of the kernel $k:B_{E'} \times B_{E'} \to \mathbb{R}$ given by 
\begin{align*}
k(e_1',e_2') := \langle \textup{cov}(X) e_1',e_2' \rangle_{E,E'}.
\end{align*}
Since we require that $E$ is separable, it follows from \cite[Lemma 8.2.3]{DW} that $H_X$ is separable.
For $j \in \mathbb{N}$, $d_j \in \mathbb{N}$ and $1 \leq m \leq d_j$ let $(v_{j,m}) \subseteq H_X$  be an ONB. Then, for all $e_1',e_2' \in B_{E'}$, we have $k(e'_1,e_2')=\sum_{j=1}^\infty \sum_{m=1}^{d_j} v_{j,m}(e'_1) v_{j,m}(e'_2)$, see \cite[Theorem 4.20]{svm}. 
The following lemma is essentail in proving the main result.
\begin{lemma}\label{lem:wichtigeungleichung}
Let $(v_{j,m}) \subseteq H_X$ be an ONB, and let $a_j > 0$ be a monotone increasing sequence, and let $c_j>0$ be a montone decreasing sequence, such that
\begin{align}\label{eq:standardbedingung}
\sup_{e' \in B_{E'}} \sum_{j=n+1}^\infty \sum_{m=1}^{d_j} v_{j,m}^2(e') \leq c_n ,
\end{align}
$a_j c_j \to 0$, and $\left(a_j \cdot (c_{j-1}-c_j) \right) \in \ell^1(\mathbb{N})$. We then have for all $e' \in B_{E'}$
\begin{align*}
\sum_{j=n+1}^\infty a_j \sum_{m=1}^{d_j} v_{j,m}(e')^2 \leq \sum_{j=n+1}^\infty a_j (c_{j-1}-c_j) < \infty.
\end{align*}
\end{lemma}
\begin{proof}
Set $\overline{P}^2_n(e'):=\sum_{j=n+1}^\infty \sum_{m=1}^{d_j} v_{j,m}^2(e')$. We then have
\begin{align*}
 \sum_{j=n+1}^\infty a_j \sum_{m=1}^{d_j} v_{j,m}^2(e')
= \sum_{j=n+1}^\infty a_j \cdot \left( \overline{P}_{j-1}^2(e') - \overline{P}_{j}^2(e') \right). 
\end{align*}
For a finite sum with $l>n+1$, we compute 
\begin{align*}
\sum_{j=n+1}^l a_j \cdot \left( \overline{P}_{j-1}^2(e') - \overline{P}_{j}^2(e') \right)  
&=\sum_{j=n+1}^l a_j \overline{P}^2_{j-1}(e')-\sum_{j=n+1}^l a_j \overline{P}_j^2(e')\\
&=a_{n+1}\overline{P}_n^2(e') - a_l \overline{P}_l^2(e') + \sum_{j=n+2}^l a_j \overline{P}_{j-1}^2(e')-\sum_{j=n+1}^{l-1} a_j \overline{P}_j^2(e') \\
&=a_{n+1}\overline{P}_n^2(e') - a_l \overline{P}_l^2(e') + \sum_{j=n+1}^{l-1} a_{j+1} \overline{P}_{j}^2(e') -\sum_{j=n+1}^{l-1} a_j \overline{P}_j^2(e') \\
&= a_{n+1}\overline{P}_n^2(e') - a_l \overline{P}_l^2(e') + \sum_{j=n+1}^{l-1} \overline{P}_j^2(e')\left(a_{j+1}-a_j \right) \\
&\leq a_{n+1}c_n - a_l \overline{P}_l^2(e') + \sum_{j=n+1}^{l-1} c_j(a_{j+1}-a_j),
\end{align*}
where in the last step we used  the monotonicity of $(a_j)$.
By applying summation of parts on the right hand side we obtain 
\begin{align*}
a_{n+1}c_n - a_l \overline{P}_l^2(e') + \sum_{j=n+1}^{l-1} c_j(a_{j+1}-a_j) 
&= a_{n+1}c_n - a_l \overline{P}_l^2(e') + \sum_{j=n+1}^{l-1} c_j a_{j+1} - \sum_{j=n+1}^{l-1} c_j a_j \\
&= a_{n+1}c_n - a_l \overline{P}_l^2(e') + \sum_{j=n+2}^l     c_{j-1} a_j - \sum_{j=n+1}^{l-1} c_j a_j  \\
&= c_l a_l - c_l a_l - a_l \overline{P}_l^2(e') + \sum_{j=n+1}^l     c_{j-1} a_j - \sum_{j=n+1}^{l-1} c_j a_j  \\
&= c_l a_l  - a_l \overline{P}_l^2(e')  + \sum_{j=n+1}^l     c_{j-1} a_j - \sum_{j=n+1}^{l} c_j a_j \\
&= c_l a_l  - a_l \overline{P}_l^2(e')  + \sum_{j=n+1}^l     a_j \left(c_{j-1} - c_j \right).
\end{align*}
Since $0 \leq a_l \overline{P}_l^2(e') \leq a_l c_l$ and by assumption $a_lc_l \to 0$ we have $a_l \overline{P}_l^2(e') \to 0$. Taking the limit $l \to \infty$ gives
\begin{align*}
\sum_{j=n+1}^\infty a_j \cdot \left( \overline{P}_{j-1}^2(e') - \overline{P}_{j}^2(e') \right)   \leq  \sum_{j=n+1}^\infty a_j \left(c_{j-1}-c_j \right).
\end{align*}
\end{proof}
Note that the key assumption in Lemma \ref{lem:wichtigeungleichung} is \eqref{eq:standardbedingung}, which, as we will see later, implies a convergence rate of the posterior variance.
 Moreover, \eqref{eq:standardbedingung} with the conditions that there exist a monotone increasing sequence $a_j>0$ with $(a_j (c_{j-1}-c_j)) \in \ell^1(\mathbb{N})$ and $a_j c_j \to 0$ allows us to define a kernel $k_{a,V,d}:B_{E'} \times B_{E'} \to \mathbb{R}$ by
\begin{align*}
k_{a,V,d}(e_1',e_2') := \sum_{j=1}^\infty a_j \sum_{m=1}^{d_j} v_{j,m}(e'_1)v_{j,m}(e'_2),
\end{align*}  
see \cite[Lemma 4.2]{svm}. We denote the associated RKHS by $H_{a,V,d}$.
Furthermore, we recall $H_X \subseteq H_{a,V,d}$, see \cite[Proposition 3.3]{samplespace}, and that $(\sqrt{a_j} v_{j,m}) \subset H_{a,V,d}$ is an ONB, see \cite[Proposition 3.3]{samplespace}

Let $V_n:=\textup{span} \{ v_{j,m} \, | \, 1 \leq j\leq n, \, 1 \leq m \leq d_j \}$ and $\Pi_{n,X}:H_{X} \to H_{X}$ denote the orthogonal projection onto $V_n$ in $H_X$. Similarly, let $\Pi_n: H_{a,V,d} \to H_{a,V,d}$ be the orthogonal projection onto $V_n$ in $H_{a,V,d}$. It follows that these two projections coincide on $V_n$, as the following lemma shows.

\begin{lemma}
Let $\Pi_{n,X}: H_X \to H_X$ be the orthogonal projection in $H_X$ onto $V_n:= \textup{span} \{ v_{j,m} \, | \,1 \leq j\leq n, \, 1 \leq m \leq d_j \}$ and $\Pi_n: H_{a,V,d} \to H_{a,V,d}$ be the orthogonal projection in $H_{a,V,d}$ onto the same $V_n$. Then, for all $f \in H_X \subseteq H_{a,V,d}$, we have
\begin{align*}
\Pi_{n,X} f = \Pi_n f.
\end{align*}
\end{lemma}
\begin{proof}
For $f \in H_X \subseteq H_{a,V,d}$ there exists a sequence $(b_{j,m})$ with $\sum_{j=1}^\infty \sum_{m=1}^{d_j} b_{j,m}^2 < \infty$ such that $f=\sum_{j=1}^\infty \sum_{m=1}^{d_j} b_{j,m} v_{j,m}$. Since $(\sqrt{a_j}v_{j,m})$ is an ONB of $H_{a,V,d}$. We conclude
\begin{align*}
\Pi_{n,X} f = \sum_{j=1}^n \sum_{m=1}^{d_j} b_{j,m} v_{j,m} = \sum_{j=1}^n \sum_{m=1}^{d_j} \frac{b_{j,m}}{\sqrt{a_j}} (\sqrt{a_j} v_{j,m}) = \Pi_n f.
\end{align*}
\end{proof}

It is essential that the orthogonal projections $\Pi_{n,X}$ and $\Pi_n$ coincide. In the context of interpolation, computing $\Pi_{n,X}$ involves inverting the kernel matrix. Since the projections are the same, we can use the same kernel matrix to compute $\Pi_n$. Consequently, the precise values of $k_{a,V,d}$ are never needed.

The next lemma is closely related to the escaping-the-native-space problem and is essential for proving our main result.

\begin{lemma}\label{lem:inequality mainMehr}
Let $(v_{j,m}) \subseteq H_X$ be an ONB, and let $a_j > 0$ be a monotone increasing sequence. Moreover, let $c_j>0$ be a monotone decreasing sequence, such that \eqref{eq:standardbedingung} is satisfied and we have both $a_j c_j \to 0$ and $\left(a_j \cdot (c_{j-1}-c_j) \right) \in \ell^1(\mathbb{N})$. Then, for all $f\in H_{a,V,d}$ the following inequality holds
\begin{align*}
\Vert f - \Pi_n f \Vert_{E''}^2 \leq \sum_{j=n+1}^\infty a_j (c_{j-1}-c_j) \cdot \Vert f - \Pi_n f  \Vert_{H_{a,V,d}}^2,
\end{align*}
where $\Pi_n$ denotes the orthogonal projection onto $V_n:=\textup{span} \{ v_{j,m} \, | \, 1\leq  j\leq n, \, 1 \leq m \leq d_j \}$ in $H_{a,V,d}$.
\end{lemma}
\begin{proof}
Let $\textup{Id}:H_{a,V,d} \to H_{a,V,d}$ denote the identity mapping. By the reproducing property of $H_{a,V,d}$ and the fact that $\Pi_n$ is an orthogonal projection, which we use in the second and third equation, we have
\begin{align*}
(f- \Pi_n f )(e') = \langle k_{a,V,d}(\mycdot,e') , f- \Pi_n f \rangle_{H_{a,V,d}}  
&=\langle k_{a,V,d}(\mycdot,e') , (\textup{Id}-\Pi_n)^2  f \rangle_{H_{a,V,d}} \\
&= \langle k_{a,V,d}(\mycdot,e') - \Pi_n k_{a,V,d}(\mycdot,e') , f- \Pi_n f \rangle_{H_{a,V,d}}.
\end{align*}
Applying the Cauchy-Schwarz inequality gives 
\begin{align*}
| (f- \Pi_n f )(e') | \leq \Vert k_{a,V,d}(\mycdot,e') - \Pi_n k_{a,V,d}(\mycdot,e') \Vert_{H_{a,V,d}} \Vert  f- \Pi_n f \Vert_{H_{a,V,d}}. 
\end{align*}
Set $\overline{P}^2_n(e'):=\sum_{j=n+1}^\infty \sum_{m=1}^{d_j} v_{j,m}^2(e')$. By applying \cite[Lemma 2.3]{pgreedy} to the kernel $k_{a,V,d}$, we have
\begin{align*}
\Vert k_{a,V,d}(\mycdot,e') - \Pi_n k_{a,V,d}(\mycdot,e') \Vert_{H_{a,V,d}}^2 
= \sum_{j=n+1}^\infty a_j \sum_{m=1}^{d_j} v_{j,m}^2(e').
\end{align*}
By Lemma \ref{lem:wichtigeungleichung} we obtain 
\begin{align*}
\sum_{j=n+1}^\infty a_j \sum_{m=1}^{d_j} v_{j,m}(e')^2 \leq \sum_{j=n+1}^\infty a_j (c_{j-1} -c_j).
\end{align*}
Finally, putting everything together and taking $\sup_{e'\in B_{E'}}$ leads to the assertion.
\end{proof}

We note that in \eqref{eq:ConditioningMehr} we have a sequence $(e'_{j,m}) \subset E'$, and by \cite[Sections 2.2 and 2.3]{daniel} we have
\begin{align}\label{eq:hjm}
h_{j,m}:=\int_\Omega e'_{j,m}(X) \iota X \, \textup{d} \mu \in H_X,
\end{align}
where $\iota: E \to E''|_{B_E'}$ denotes the canonical embedding into the bidual space, which is then restricted to the unit ball. We use the sequence $(e'_{j,m})$ to construct an orthonormal system in $H_X$. This leads to the following assumption.
\begin{assumptionV}\label{ass:V}
Given a sequence $(e_{j,m}') \subseteq E'$, we define $h_{j,m}$ by \eqref{eq:hjm} and let $(v_{j,m}) \subseteq H_X$ be an ONB of $\overline{\textup{span}\{h_{j,m}\}} \subseteq H_X$ satisfying
\begin{align*}
\textup{span}\{v_{j,m} \, | 1 \leq j \leq n, \, 1 \leq m \leq d_j \}= \textup{span}\{h_{j,m} \, | 1 \leq j \leq n , \,  1 \leq m \leq d_j \}
\end{align*}
 for all $n$.
\end{assumptionV}

\begin{remark}
Under \hyperref[ass:V]{Assumption V}, if $\textup{span} \{ e_{j,m}' \, | \, j \in \mathbb{N}, \, 1 \leq m \leq d_j\}$ is $\textup{weak}^*$-dense in $E'$, then by \cite[Lemma 8.2.3]{DW}, the family $(v_{j,m})$ forms an ONB of $H_X$.
\end{remark}

\begin{lemma}\label{lem:ONB}
If \hyperref[ass:V]{Assumption V} is satisfied, then there exist independent $\mathcal{N}(0,1)$ distributed random variables $r_{j,m} :\Omega \to \mathbb{R}$ such that $\int_\Omega r_{j,m} \iota X \, \textup{d} \mu = v_{j,m}$ and
\begin{align*}
 \iota \Xbn = \sum_{j=1}^n \sum_{m=1}^{d_j} r_{j,m} v_{j,m}.
\end{align*} 
Moreover, if $ \Vert \textup{cov}(X) - \textup{cov}(\Xbn) \Vert_{E' \to E} \to 0$, we also have 
\begin{align*}
\iota X = \sum_{j=1}^\infty \sum_{m=1}^{d_j} r_{j,m} v_{j,m}.
\end{align*}
\end{lemma}
\begin{proof}

Our first goal is to contruct the random variables $r_{j,m}$.
To this end let $(v_{j,m})$ be an ONB of $V_n:= \textup{span} \{ \int_\Omega e_{j,m}'(X) \iota X \, \textup{d} \mu \, |  \, 1 \leq j \leq n, 1 \leq m \leq d_j \}$ and $(v_j^\perp)$ be an ONB of $V^\perp:=(\cup_{n \in \mathbb{N}} V)^\perp$. We note that since $H_X$ is separable $(v_{j,m})$ and $(v_j^\perp)$ are both countable.
Furthermore, $(v_{j,m}) \cup (v_j^\perp)$ is an ONB of $H_X$ and for all $e'\in B_{E'}$ we have
\begin{align}\label{eq:covX}
\textup{cov}(X)e'= k(e', \mycdot)=\sum_{j=1}^\infty  \sum_{m=1}^{d_j}v_{j,m}(e') v_{j,m}(\mycdot) + \sum_{j=1}^\infty v_j^\perp (e') v_j^\perp(\mycdot).
\end{align}
By \cite[Theorem 4]{daniel}, for all $e' \in B_{E'}$ we also have
\begin{align}\label{eq:covXbn}
\textup{cov}(\Xbn)e'= \sum_{j=1}^n \sum_{m=1}^{d_j} v_{j,m}(e') v_{j,m}(\mycdot).
\end{align}
By \cite[Theorem 3.1]{RolandOpfer}, we have that $(v_{j,m}) \subset H_{\Xbn}$ is a Parseval frame, that is 
\begin{align*}
\sum_{j=1}^n \sum_{m=1}^{d_j}| \langle f , v_{j,m} \rangle_{H_{\Xbn}} |^2 = \Vert f \Vert_{H_{\Xbn}}^2
\end{align*}
 for all $f \in H_{\Xbn}$. Consequently we have $H_{\Xbn}= \textup{span}\{ v_{j,m} \, | \, 1 \leq j \leq m, 1 \leq m \leq d_j\}$ by \eqref{eq:covXbn} and \cite[Theorem 4.21]{svm} and since the vectors $(v_{j,m})$ are linearly independent, they are a basis of $H_{\Xbn}$.
Moreover, by \cite[Proposition 1.11]{finiteFrames}, we have, for all $i,l \in \mathbb{N}$,
\begin{align*}
v_{i,l} = \sum_{j=1}^n \sum_{m=1}^{d_j} \langle v_{i,l} , v_{j,m} \rangle_{H_{\Xbn}} v_{j,m}.
\end{align*}
Since $(v_{j,m})$ is a basis, the coefficients are unique. We conclude $\langle v_{i,l} , v_{j,m} \rangle_{H_{\Xbn}} = 1$ if $(j,m)=(i,l)$ and $0$ else. In summary, $(v_{j,m})$ is an ONB of $H_{\Xbn}$.

We now define $G_{\Xbn}:=\overline{\{ e'(\Xbn) \, | \, e' \in E'\}}^{\Vert \mycdot \Vert_{L^2(\mu)} }$ and 
$W_{\Xbn}:= \{ \int_\Omega g \Xbn \, \textup{d} \mu \, | \, g \in G_{\Xbn}\}$ with 
$\Vert f \Vert_{W_{\Xbn}} := \inf_{g \in G_{\Xbn}} \Vert g \Vert_{L^2(\mu)}$. 
By Lemma \cite[Lemma 15]{daniel} the mapping $V_{\Xbn}:(G_{\Xbn}, \Vert \mycdot \Vert_{L^2(\mu)}) \to (W_{\Xbn}, \Vert \mycdot \Vert_{W_{\Xbn}})$ defined by 
\begin{align*}
V_{\Xbn} g := \int_\Omega g \Xbn \, \textup{d} \mu
\end{align*}
is an isometry. Furthermore, the mapping $\iota : W_{\Xbn} \to H_{\Xbn}$ is an isometry by \cite[Section 2.3]{daniel}. Thus we define $ r_{j,m}:= V_{\Xbn}^{-1} \iota^{-1}  v_{j,m}$. Furthermore, we have $r_{j,m} \in G_{\Xbn}$ and by definition they are normalized and orthogonal, thus $r_{j,m} \sim \mathcal{N}(0,1)$ i.i.d.. 
By \cite[Theorem 3]{daniel}, we have 
\begin{align*}
\Xbn = \sum_{j=1}^n \sum_{m=1}^{d_j} r_{j,m} v_{j,m} .
\end{align*} 

Let us now assume that $\textup{cov}(\Xbn) \to \textup{cov}(X)$ in the operator norm. Our first goal is to show
\begin{align}\label{eq:Dichtheit}
\overline{\textup{span} \left\{\int_\Omega e_{j,m}'(X) \iota X \, \textup{d} \mu \right\}}^{\Vert \cdot \Vert_{H_X}} = H_X.
\end{align}
It suffices to show $V^{\perp} = \{ 0 \}$. 
By assumption, $\textup{cov}(X_n)e' \to \textup{cov}(X)e'$, so \eqref{eq:covX} and \eqref{eq:covXbn} give
\begin{align*}
\textup{cov}(X)e' - \textup{cov}(\Xbn)e' = \sum_{j=1}^\infty v_j^\perp (e') v_j^\perp(\mycdot)  + \sum_{j=n+1}^\infty \sum_{m=1}^{d_j} v_{j,m}(e') v_{j,m}(\mycdot) \to 0.
\end{align*}
This implies $\sum_{j=1}^\infty v_j^\perp (e') v_j^\perp(e') + \sum_{j=n+1}^\infty \sum_{m=1}^{d_j}  v_{j,m}(e') v_{j,m}(e') \to 0$ and since the first sum is independent of 
$n$, we conclude that $\sum_{j=1}^\infty v_j^\perp (e') v_j^\perp(e')=0$ for all $e' \in E'$. 
This implies $v_j^\perp=0$ for all $j \in \mathbb{N}$, so $V^\perp=\{0\}$. Therefore \eqref{eq:Dichtheit} holds true. 

By \cite[Theorem 3.3.2]{martingal} we have
\begin{align*}
\sum_{j=1}^n \sum_{m=1}^{d_j} r_{j,m} v_{j,m} \to \sum_{j=1}^\infty \sum_{m=1}^{d_j} r_{j,m} v_{j,m},
\end{align*}
where the convergence is almost surely.
Finally, $X - \Xbn$ is Gaussian with mean zero and covariance $\textup{cov}(X-\Xbn)$. By \cite[Lemma 30]{daniel} we have
\begin{align*}
\textup{cov}(X-\Xbn) = \textup{cov}(X)-\textup{cov}(\Xbn),
\end{align*} 
which converges to $0$ by assumption, thus the limit of $\Xbn$ is truly $X$ and we conclude 
\begin{align*}
\iota X = \sum_{j=1}^\infty \sum_{m=1}^{d_j} r_{j,m} v_{j,m}.
\end{align*}
\end{proof}
\subsection*{Concentration Inequality}
\begin{lemma}\label{lem:Concentration}
Let $(r_j) \sim \mathcal{N}(0,1)$ be i.i.d. and let $(b_j) \in \ell^1(\mathbb{N})$ with $b_j \geq 0$ for all $j \geq 1$. Then the series $\sum_{j=1}^\infty b_j r_j^2$ converges almost surely, and for $Z:=\sum_{j=1}^\infty b_j(r_j^2-1)$ we have, for all $\tau>0$ 
 \begin{align*}
\mu(Z \geq 2 \Vert (b_j) \Vert_{\ell^2} \sqrt{\tau}+2 \Vert (b_j)\Vert_{\ell^\infty}  \tau )\leq \mathrm{e}^{-\tau}.
\end{align*}
\end{lemma}
\begin{proof}
First, we show that $\sum_{j=1}^\infty b_j r_j^2$ converges almost surely. Since $b_j r_j^2 \ge 0$, we can apply Fubini's theorem to calculate the expectation
\begin{align*}
\mathbb{E} \sum_{j=1}^\infty b_j r_j^2  = \sum_{j=1}^\infty b_j \mathbb{E}r_j^2 = \sum_{j=1}^\infty b_j = \Vert b_j \Vert_{\ell^1}< \infty.
\end{align*}
A series of nonnegative random variables with finite expectation converges almost surely, so the series converges.
For the concentration inequality, we follow the argument from \cite[Lemma 1]{Massart1}. For $x < 1/2$, \cite[Theorem 4.2.3]{chisquared} gives
\begin{align*}
\psi(x):= \ln\left( \mathbb{E}  \mathrm{e}^{x(r_j^2-1)}  \right) = \ln\left( \mathrm{e}^{-x} \mathbb{E} \mathrm{e}^{r_j^2}  \right) = -x - \frac{1}{2} \ln(1-2x).
\end{align*}
Using Lemma \ref{lem:einfacheUngleichung}, for $0 < x < 1/2$ we have
\begin{align*}
0<\psi(x) \leq \frac{x^2}{1-2x}.
\end{align*}
 Thus, for $0<x<1/ (2 \Vert (b_j) \Vert_{\ell^\infty})$, the independence of $(r_j)$ yields
\begin{align*}
\ln\left(\mathbb{E}  \mathrm{e}^{xZ}  \right) = \sum_{j=1}^\infty \ln \left( \mathbb{E} \mathrm{e}^{b_j x(r_j^2-1)} \right) \leq  \sum_{j=1}^\infty \frac{b_j^2x^2}{1-2b_jx} \leq \frac{x^2\Vert (b_j) \Vert_{\ell^2}^2 }{1-2  x\Vert (b_j)\Vert_{\ell^\infty} } < \infty.
\end{align*} 
All series converge since the terms are nonnegative. 
Finally, applying Lemma \ref{lem:concentration inequality} with $v = 2 \Vert (b_j) \Vert_{\ell^2}^2$ and $c = 2 \Vert (b_j)\Vert_{\ell^\infty}$ gives the stated bound.
\end{proof}

\begin{corollary}\label{cor:Concentration}
Under \hyperref[ass:V]{Assumption V}, let $(a_j)$ be a sequence with $a_j>0$ for all $j \in \mathbb{N}$ and $(d_j/a_j) \in \ell^1(\mathbb{N})$. Then have $\iota X,\iota \Xbn \in H_{a,V}$ almost surely, and 
\begin{align*}
\mu \left(\Vert \iota X - \iota \Xbn \Vert_{H_{a,V,d}}^2 \leq  5 \max \{1,\tau\} \sum_{j=n+1}^\infty \frac{d_j}{a_j} \right)  \geq 1-\mathrm{e}^{-\tau}.
\end{align*} 
\end{corollary}

\begin{proof}
By Lemma \ref{lem:ONB}, we can write $\iota X=\sum_{j=1}^\infty \sum_{m=1}^{d_j} r_{j,m} v_{j,m}$. Applying Lemma \ref{lem:Concentration}, we get
\begin{align*}
\sum_{j=1}^\infty \frac{1}{a_j} \sum_{m=1}^{d_j} r_{j,m}^2 < \infty
\end{align*}
almost surely. Since $(\sqrt{a_j}v_{j,m})$ is an ONB of $H_{a,V,d}$ we conclude $\iota X \in H_{a,V,d}$. Moreover, by Lemma \ref{lem:ONB}, we have $\iota X-\iota \Xbn = \sum_{j=n+1}^\infty \sum_{m=1}^{d_j} r_{j,m} v_{j,m}$. Applying the concentration inequality from Lemma \ref{lem:Concentration}, we obtain for any $\tau>0$
\begin{align*}
\mu\left( \Vert \iota X- \iota X_n \Vert_{H_{a,V}}^2 \geq \sum_{j=n+1}^\infty d_j/a_j +2 \sqrt{\tau} \sqrt{\sum_{j=n+1}^\infty d_j^2/a_j^2} + 2 \tau \sup_{j \geq n+1} d_j/a_j \right) \leq \mathrm{e}^{-\tau}.
\end{align*} 
Furthermore, note that we have $\Vert \left( d_{j+n}/a_{j+n} \right) \Vert_{\ell^\infty} \leq \Vert \left( d_{j+n}/a_{j+n} \right) \Vert_{\ell^2} \leq \Vert \left( d_{j+n}/a_{j+n} \right) \Vert_{\ell^1}$ and thus we find 
\begin{align*}
 &\left\Vert \left( d_{j+n}/a_{j+n} \right) \right\Vert_{\ell^1(\mathbb{N})} + 2 \sqrt{\tau} \left\Vert \left( d_{j+n}/a_{j+n} \right) \right\Vert_{\ell^2(\mathbb{N})} +2 \tau \left\Vert \left( d_{j+n}/a_{j+n} \right) \right\Vert_{\ell^\infty(\mathbb{N})} \\
  \leq &5 \max\{1,\tau\} \left\Vert \left( d_{j+n}/a_{j+n} \right)\right\Vert_{\ell^1(\mathbb{N})}.
\end{align*} 
Putting everything together leads to the assertion.
\end{proof}

\subsection*{Proof of Main Result and Corollaries}

\begin{proof}[Proof of Theorem \ref{Theorem:MehrMain Result}]
Let $\left( (v_{j,m})_{m=1}^{d_j} \right)_{j=1}^n \subseteq H_{\Xbn}$ be an ONB of $H_{\Xbn}$. We then have
\begin{align*}
 \sup_{e' \in B_{E'} } \sum_{j=n+1}^\infty \sum_{m=1}^{d_j} v_{j,m}^2(e') = \Vert \textup{cov}(X)-\textup{cov}(\Xbn) \Vert_{E'\to E}  \leq c_n.
\end{align*} 
Moreover, we have $\Vert \iota X - \iota \Xbn \Vert_{E''} = \Vert X - \Xbn \Vert_E$.
By Lemma \ref{lem:ONB} we have $\Xbn=\Pi_n X$. Applying Lemma \ref{lem:inequality mainMehr}, we obtain
\begin{align*}
\Vert X - \Xbn \Vert_E = \Vert \iota X - \iota \Xbn \Vert_{E''} \leq \sqrt{\sum_{j=n+1}^\infty a_j (c_{j-1}-c_j) \cdot \Vert \iota X - \iota \Xbn \Vert_{H_{a,V,d}}^2}.
\end{align*}
Finally, using Corollary \ref{cor:Concentration} leads to the assertion.
\end{proof}
\begin{proof}[Proof of Corollary \ref{Corollary:PolynomHigher}]
We will apply Theorem \ref{Theorem:MehrMain Result} with $a_j:=j^{\gamma}$ for $\beta+1 < \gamma < \alpha$. 
The condition $\beta+1< \gamma$ implies $(d_j/a_j) \in \ell^1(\mathbb{N})$. We note $(a_j)$ is monotone increasing. Furthermore, $\gamma < \alpha$ implies $a_j c_j \to 0$.

Next, we show that $\left(C j^{-\gamma}(j^{-\alpha} - (j+1)^{-\alpha}) \right) \in \ell^1(\mathbb{N})$. 
To this end we define $f:(0,\infty) \to \mathbb{R}$ by $f(t):=C(t+1)^{-\alpha}$. By the mean value theorem there exist $\xi_j \in [j-1,j]$ such that $c_{j-1}-c_j=-f'(\xi_j)=C \alpha(\xi_j+1)^{-\alpha-1} \leq C \alpha j^{-\alpha-1}$. Thus $ \left( a_j (c_{j-1}-c_j) \right) \in \ell^1(\mathbb{N})$ is equivalent to $\gamma-\alpha-1 <-1$ meaning $\gamma< \alpha$. In other words the assumption of Theorem \ref{Theorem:MehrMain Result} are satisfied.

Now we estimate 
\begin{align*}
\sum_{j=n+1}^\infty a_j (c_{j-1}-c_j) = \sum_{j=n+1}^\infty -a_j f'(\xi_j) = \sum_{j=n+1}^\infty C j^\gamma \alpha (\xi_j+1)^{-\alpha-1} \leq C\alpha \sum_{j=n+1}^\infty j^{\gamma-\alpha-1}.
\end{align*}
Using that $j^{\gamma-\alpha-1}$ is decreasing, we bound the sum by an integral
\begin{align*}
\sum_{j=n+1}^\infty j^{\gamma-\alpha-1}   \leq \int_n^\infty t^{\gamma-\alpha-1} \textup{d} t = \left[ \frac{t^{\gamma-\alpha}}{\gamma-\alpha} \right]_{t=n}^\infty = \frac{n^{\gamma-\alpha}}{\alpha-\gamma}.
\end{align*} 
Similarly, for $(d_j / a_j)$, we have
\begin{align*}
\sum_{j=n+1}^\infty \frac{d_j}{a_j}
\leq  C_d 2^\beta \sum_{j=n+1}^\infty j^{\beta-\gamma}
\leq  C_d 2^\beta \int_n^\infty t^{\beta-\gamma} \textup{d} t 
=  C_d 2^\beta \left[ \frac{t^{\beta+1-\gamma}}{\beta+1-\gamma} \right]_{t=n}^\infty
=  C_d 2^\beta \frac{n^{\beta+1-\gamma}}{\gamma-\beta-1}.
\end{align*} 
Multiplying the two terms gives
\begin{align*}
\left( C \sum_{j=n+1}^\infty j^{\gamma-\alpha-1} \right) \left( \sum_{j=n+1}^\infty \frac{d_j}{a_j} \right) \leq  C C_d 2^\beta \frac{n^{\gamma-\alpha}}{\alpha-\gamma} \frac{n^{\beta+1-\gamma}}{\gamma-\beta-1} =  C C_d 2^\beta \frac{n^{-\alpha+\beta+1}}{(\alpha-\gamma)(\gamma-\beta-1)}.
\end{align*}
By setting $\gamma:=\frac{\alpha+\beta+1}{2}$ and putting everything together we obtain by Theorem \ref{Theorem:MehrMain Result}
\begin{align*}
\mu \left(\Vert X-X^{(N)} \Vert_{E' \to E} \leq \frac{\sqrt{20  C C_d 2^\beta \max\{1,\tau \}}}{\alpha-\beta-1} n^{\frac{-\alpha+\beta+1}{2}} \right)> 1 - \mathrm{e}^{-\tau}.
\end{align*}
\end{proof}

\begin{proof}[Proof of Theorem \ref{Theorem:simple Main Result}]
The result follows directly from Theorem \ref{Theorem:MehrMain Result} by setting $d_j = 1$ and $a_j:=1/\sqrt{c_{j-1}-c_j}$ for all $j \in \mathbb{N}$.
\end{proof}

\begin{proof}[Proof of Corollary \ref{Corollary:Polynom}]
This follows directly by Corollary \ref{Corollary:PolynomHigher} with $d_j=1$.
\end{proof}

\begin{proof}[Proof of Corollary \ref{Corollary:exponent}]
We will apply Theorem \ref{Theorem:simple Main Result} with $c_n:=C_1 \mathrm{e}^{-C_2 n^{1/\alpha}}$. First, we note that $(c_n) \in \ell^1(\mathbb{N})$ and $(c_n)$ is positive monotone decreasing.  
To show that $(c_{j-1}-c_j)$ is monotone decreasing, observe that this is equivalent to $c_{j-1}-c_j - (c_j - c_{j+1}) \geq 0$. Applying the mean value theorem twice to the function $f: [0,\infty) \to \mathbb{R}$ $f(t):= C_1 \mathrm{e}^{-C_2 t^{1/\alpha}}$, there exist $\xi \in [j-1,j]$,$\zeta \in[j,j+1]$ such that 
\begin{align*}
c_{j-1}-c_j-(c_j -c_{j+1}) = -f'(\xi)+f'(\zeta).
\end{align*}
Since $f$ is convex and $\zeta > \xi$, we conclude $f'(\zeta) - f'(\xi) \geq 0$. 

By the mean value theorem, there also exist $\xi_j \in [j-1,j]$ such that
\begin{align*}
\sum_{j=1}^\infty \left(\sqrt{c_{j-1}-c_j} \right) 
= \sum_{j=1}^\infty \sqrt{-f'(\xi_j)}
 = \sum_{j=1}^\infty \sqrt{\frac{C_1 C_2}{\alpha}} \xi_j^{\frac{1}{2\alpha}-\frac12} \mathrm{e}^{-\frac{C_2}{2} \xi_j^{1/\alpha}} 
\leq \sum_{j=1}^\infty \sqrt{\frac{C_1 C_2}{\alpha}} j^{\frac{1}{2\alpha}-\frac12} \mathrm{e}^{-\frac{C_2}{2} (j-1)^{1/\alpha}}.
\end{align*} 
The final sum is finite because of the exponential decay, hence $(\sqrt{c_{j-1}-c_j}) \in \ell^1(\mathbb{N})$.

Finally, we show that $c_j/\sqrt{c_{j-1}-c_j} \to 0$. Again, using the mean value theorem, for $j > 1$:
\begin{align*}
\frac{c_j}{\sqrt{c_{j-1}-c_j}}
= \frac{C_1 \mathrm{e}^{-C_2 j^{1/\alpha}}}{\sqrt{\frac{C_1 C_2}{\alpha}} \xi_j^{\frac{1}{2\alpha}-\frac12} \mathrm{e}^{-\frac{C_2}{2} \xi_j^{1/\alpha}} }
&\leq \frac{\sqrt{\alpha C_1 }}{\sqrt{C_2} (j-1)^{\frac{1}{2\alpha}-\frac12}} \cdot \frac{ \mathrm{e}^{-C_2 j^{1/\alpha}}}{  \mathrm{e}^{-\frac{C_2}{2} \xi_j^{1/\alpha}} } \\
&= \frac{\sqrt{\alpha C_1 }}{\sqrt{C_2} (j-1)^{\frac{1}{2\alpha}-\frac12}} \cdot \mathrm{e}^{-\frac{C_2}{2} j^{1/\alpha} } \to 0.
\end{align*} 
Thus all conditions of Theorem \ref{Theorem:simple Main Result} are satisfied. For $\alpha>1$ applying Corollary \ref{cor:gaussian} gives the desired assertion.

For the case $\alpha=1$ we calculate 
\begin{align*}
\sum_{j=n+1}^\infty \sqrt{c_{j-1}-c_j} 
= \sum_{j=n+1}^\infty \sqrt{C_1 \cdot (\mathrm{e}^{C_2})^{-(j-1)}(1-\mathrm{e}^{-C_2})} 
&= \sqrt{C_1 \cdot (1 - \mathrm{e}^{-C_2}) \mathrm{e}^{C_2}} \sum_{j=n+1}^\infty (\sqrt{\mathrm{e}^{C_2}})^{-j} \\
&= \sqrt{C_1 \cdot (\mathrm{e}^{C_2} - 1) } (\mathrm{e}^{C_2})^{-\frac{n+1}{2}} \sum_{j=0}^\infty (\sqrt{\mathrm{e}^{C_2}})^{-j} \\
&= \sqrt{C_1 \cdot (\mathrm{e}^{C_2} - 1) } \cdot \frac{\mathrm{e}^{C_2}}{\sqrt{\mathrm{e}^{C_2}}-1} \mathrm{e}^{-\frac{C_2}{2} n}. 
\end{align*}
With Theorem \ref{Theorem:simple Main Result} we obtain the assertion.
\end{proof}
\begin{lemma}\label{lem:stetigeNorm}
If $E=C(T)$ and $X$ is an $E$-valued Gaussian random variable, we have 
\begin{align*}
\sup_{t \in T} | k_X(t,t) - k_{\Xbn}(t,t) | = \Vert \textup{cov}(X) - \textup{cov}(\Xbn) \Vert_{C(T)' \to C(T)}.
\end{align*} 
\end{lemma}
\begin{proof}
We set $Z=X-\Xbn$ and show that
\begin{align*}
\sup_{t \in T} | k_Z(t,t) | = \Vert \textup{cov}(Z) \Vert_{C(T)' \to C(T) }
\end{align*}
holds true for all $E$-valued Gaussian random variables $Z$. First, recall that 
\begin{align*}
\sup_{t \in T} |k_Z(t,t)| = \sup_{t,s \in T} | k_Z(t,s)| = \sup_{t,s \in T} \langle \textup{cov}(Z) \delta_t,\delta_s \rangle_{C(T),C(T)'},
\end{align*}
where $\delta_t \in C(T)'$ denotes the point evaluation at $t$. Furthermore, by Krein-Milman theorem we obtain $\overline{\textup{aco}\{ \delta_t \, | \, t \in T\}}^{w^*} = B_{C(T)'}$, see for instance \cite[Theorem 3.23]{Rudin}. Moreover, by \cite[Theorem B.3]{IngoConditioning} the kernel $k_Z$ is $w^*$-continuous and thus we find 
\begin{align*}
\sup_{t,s \in T} \langle \textup{cov}(Z) \delta_t,\delta_s \rangle_{E,E'}= \sup_{e_1',e_2' \in B_{C(T)'}} \langle \textup{cov}(Z) e_1',e_2' \rangle_{E,E'}.
\end{align*}
In summary we obtain
\begin{align*}
 \sup_{t \in T} |k_Z(t,t)|=\Vert \textup{cov}(Z) \Vert_{C(T)' \to C(T)}.
\end{align*}
By \cite[Lemma 30]{daniel}, we have 
\begin{align*}
k_Z=\textup{cov}(X-\Xbn) = \textup{cov}(X)-\textup{cov}(\Xbn)=k_X-k_{\Xbn}.
\end{align*}
\end{proof}

\begin{remark}
We obtain Equation \eqref{eq:StetigeNorm} by setting $d_j=1$ for all $j \in \mathbb{N}$ in Lemma \ref{lem:stetigeNorm}. Additionally, setting $n=0$ in Lemma \ref{lem:stetigeNorm} yields $\Vert \textup{cov}(X) \Vert_{C(T)' \to C(T)} = \sup_{t \in T} | k_X(t,t) |$.
\end{remark}  

\begin{lemma}\label{lem:Eigenvalue matches Norm}
Let $X$ be a Gaussian random variable, $E=L^2(\lambda)$ and denote by $\lambda_j$ the monotone decreasing eigenvalues of $\textup{cov}(X):L^2(\lambda) \to L^2(\lambda)$. Furthermore, we denote by $H_j$ the eigenspace associated to $\lambda_j$, we write $d_j:=\textup{dim}(H_j)$, and let $(e_{j,m})_{m=1}^{d_j} \subseteq H_j$ be an ONB of $H_j$. For $(e_{j,m}') = \langle e_{j,m}, \mycdot \rangle \in L^2(\lambda)'$ we then have
\begin{align*}
\Xbn= \sum_{j=1}^n \sqrt{\lambda_j} \sum_{m=1}^{d_j} r_{j,m} e_{j,m} \quad \textup{and} \quad \Vert \textup{cov}(X)-\textup{cov}(\Xbn) \Vert_{L^2(\lambda) \to L^2(\lambda)} = \lambda_{n+1}.
\end{align*}
\end{lemma}
\begin{proof}
Because $\textup{cov}(X)$ is a symmetric, nuclear operator, see \cite{CovarianceSymmetric}, we can write 
\begin{align*}
\textup{cov}(X) f = \sum_{j=1}^\infty \lambda_j \sum_{m=1}^{d_j} \langle f , e_{j,m} \rangle_{L^2(\lambda)} e_{j,m}, \quad \textup{for} \, \, f \in L^2(\lambda).
\end{align*}  
By \cite[Example 45]{daniel} we know that $(\sqrt{\lambda_j} e_{j,m}) \subset H_X$ is an ONB and 
\begin{align*}
h_{j,m}:= \int_\Omega e_{j,m}'(X) \iota X \, \textup{d} \mu = \textup{cov}(X) e_{j,m} = \lambda_{j} e_{j,m}.
\end{align*}
Consequently, \hyperref[ass:V]{Assumption V} is satisfied and Lemma \ref{lem:ONB} thus shows
\begin{align*}
\Xbn= \sum_{j=1}^n \sqrt{\lambda_j} \sum_{m=1}^{d_j} r_{j,m} e_{j,m}.
\end{align*}
We conclude 
\begin{align*}
\sup_{\Vert f \Vert_{L^2(\lambda)} \leq 1 } \Vert \textup{cov}(X)  - \textup{cov}(\Xbn) \Vert_{L^2(\lambda)} = \sup_{\Vert f \Vert_{L^2(\lambda)} \leq 1 } \left\Vert \sum_{j=n+1}^\infty \lambda_j \sum_{m=1}^{d_j} \langle f , e_{j,m} \rangle_{L^2(\lambda)} e_{j,m} \right\Vert_{L^2(\lambda)}= \max_{j \geq n+1} \lambda_j.
\end{align*}
Because $\lambda_j>\lambda_{j+1}>0$, we conclude the assertion.
\end{proof}

\paragraph*{Acknowledgements} Daniel Winkle acknowledges funding by the International Max Planck Research School
for Intelligent Systems (IMPRS-IS).

\bibliographystyle{plain}
\bibliography{daniel_refs}

\begin{appendices}
\section{Appendix}\label{sec:appendix}\label{sec:6}
Here are all the technical calculations are necessary to prove all the claims, but they do not provide further insights. 
\begin{lemma}\label{lem:summezuintegral}
Given a decreasing sequence $(c_j)$ with $c_j>0$ and $f:[0,\infty) \to [0,\infty)$ a twice differentiable, convex, and monotone decreasing function, with $f(j)=c_j$ for all $j \in \mathbb{N}$, we have 
\begin{align*}
\sum_{j=n+1}^\infty \sqrt{c_{j-1}-c_j} \leq \int_{n-1}^\infty \sqrt{|f'(t)|} \, \textup{d} t \quad \textup{for all} \, \,  n\geq 1.
\end{align*}   
\end{lemma}
\begin{proof}
This is a consequence of the mean value theorem, we note that we have by the property of $f$ being monotone decreasing and convex
\begin{align*}
c_{j-1}-c_j=f(j-1)-f(j) \leq \sup_{\xi \in [j-1,j]} -f'(\xi) \leq -f'(j-1) \leq -f'(t) 
\end{align*} 
for all $t \in [j-2,j-1]$, thus 
\begin{align*}
\sum_{j=n+1}^\infty \sqrt{c_{j-1}-c_j} \leq \int_{n-1}^\infty \sqrt{|f'(t)|} \, \textup{d}t.
\end{align*}
\end{proof}

\begin{corollary}\label{cor:gaussian}
Let $C_1,C_2>0$, and $d>1$. For $n>\left( \frac{11}{C_2} (d-1) \right)^d +1$ we then have
\begin{align*}
\sum_{j=n+1}^\infty \sqrt{C_1\mathrm{e}^{-C_2 (j-1)^{1/d}} - C_1\mathrm{e}^{-C_2 j^{1/d}}} \leq \frac{11 \sqrt{C_1C_2d}}{10}   \left(\frac{C_2}{2} \right)^{d-2}  (n-1)^{\frac{d-1}{2d}}  \mathrm{e}^{-\frac{C_2}{2}(n-1)^{1/d}}.
\end{align*} 
\end{corollary}
\begin{proof}
We apply Lemma \ref{lem:summezuintegral} with $f(t):= C_1 \mathrm{e}^{-C_2 t^{1/d}}$. We have $f'(t)=-C_2C_1 \frac{t^{1/d-1}}{d} \mathrm{e}^{-C_2 t^{1/d}}$ and obtain 
\begin{align*}
 \sum_{j=n+1}^\infty \sqrt{C_1\mathrm{e}^{-C_2 (j-1)^{1/d}} - C_1\mathrm{e}^{-C_2 j^{1/d}} } \leq 
\sqrt{C_2C_1} \int_{n-1}^\infty \sqrt{\frac{t^{1/d-1}}{d}} \mathrm{e}^{-\frac{C_2}{2} t^{1/d}} \, \textup{d}t.
\end{align*}
Substituting $u=t^{1/d}$ we obtain  
\begin{align*}
\int_{n-1}^\infty \sqrt{\frac{t^{1/d-1}}{d}} \mathrm{e}^{-\frac{C_2}{2} t^{1/d}} \, \textup{d}t = \int_{(n-1)^{1/d}}^\infty \sqrt{\frac{u^{1-d}}{d}} \mathrm{e}^{-\frac{C_2}{2}u} d u^{d-1} \, \textup{d}u = \int_{(n-1)^{1/d}}^\infty \sqrt{d u^{d-1}} \mathrm{e}^{-\frac{C_2}{2} u } \, \textup{d} u.
\end{align*} 
Now substituting \( \frac{C_2}{2} u = x \) leads to 
\begin{align*}
\int_{(n-1)^{1/d}}^\infty \sqrt{d u^{d-1}} \mathrm{e}^{-\frac{C_2}{2} u} \, \mathrm{d}u 
&= \sqrt{d} \int_{\frac{C_2}{2}(n-1)^{1/d}}^\infty \left(\frac{C_2}{2} x\right)^{\frac{d-1}{2}} \mathrm{e}^{-x} \frac{2}{C_2} \, \mathrm{d}x.
\end{align*}
This integral matches the incomplete gamma function. Applying \cite[Inequality (3.5)]{IncompleteGammaFunction}, with \( B = \frac{11}{10} \), we obtain the following inequality for \( C_2 (n-1)^{1/d} > 11(d-1) \) 
\begin{align*}
\sqrt{ d \left(\frac{C_2}{2} \right)^{d-3}} \int_{\frac{C_2}{2}(n-1)^{1/d}}^\infty x^{\frac{d-1}{2}} \mathrm{e}^{-x} \, \mathrm{d}x \leq \frac{11 \sqrt{d}}{10}   \left(\frac{C_2}{2} \right)^{d-2}  (n-1)^{\frac{d-1}{2d}}  \mathrm{e}^{-\frac{C_2}{2}(n-1)^{1/d}},
\end{align*} 
thus the assertion follows.
\end{proof}

\begin{lemma}\label{lem:einfacheUngleichung}
For $0 \leq x < 1/2$ we have 
\begin{align*}
0 \leq -x -1/2 \ln(1-2x) \leq \frac{x^2}{1-2x}.
\end{align*}
\end{lemma}
\begin{proof}
We start by showing that $f(x):=  -x-1/2 \ln(1-2x) \geq 0$. We note that $f(0)=0$ and $f'(x)=-1+\frac{1}{1-2x}=\frac{2x}{1-2x} \geq 0$ for $x \in [0,1/2)$, thus we conclude $f(x) \geq 0$.

Next we prove $g(x):=   \frac{x^2}{1-2x}+x+1/2 \ln(1-2x) \geq 0$. We note that $g(0)=0$ and $g'(x)=\frac{2x^2}{(1-2x)^2} \geq 0$ for $x \in [0,1/2)$, thus the assertion follows.
\end{proof}

\begin{lemma}\label{lem:concentration inequality}
Let $Z:\Omega \to \mathbb{R}$ be a random variable with $\mathbb{E}(Z)=0$. If for fixed $v,c>0$, and all $0 < x < 1/c$ we have
\begin{align*}
\ln\left( \mathbb{E} \mathrm{e}^{xZ}  \right) \leq \frac{vx^2}{2(1-cx)},
\end{align*}
then for any $\tau>0$, we have
\begin{align*}
\mu(Z \geq c\tau +\sqrt{2v\tau} ) \leq \mathrm{e}^{-\tau}.
\end{align*}
\end{lemma} 
\begin{proof}
This proof is taken from \cite[Lemma 8]{Massart2}, we also add all the necessary technical details. We use the well known Markov inequality and obtain with the assumptions for all $x\geq 0$
\begin{align*}
\mu \left( Z \geq \varepsilon \right) = \mu\left( \mathrm{e}^{x Z } \geq \mathrm{e}^{ x \varepsilon} \right) \leq \mathbb{E} \mathrm{e}^{xZ} \mathrm{e}^{-x \varepsilon}= \mathrm{e}^{   -x \varepsilon+\ln\left(\mathbb{E}\mathrm{e}^{x Z} \right) } \leq \mathrm{e}^{  -x \varepsilon+ \frac{vx^2}{2(1-cx)}  }
\end{align*} 
for all $\varepsilon>0$. We define $h_\varepsilon:(0,\infty) \to \mathbb{R}$ by
\begin{align*}
h_\varepsilon (x):= \left( x \varepsilon - \frac{v x^2}{2(1-cx)} \right).
\end{align*}
Next we show that the function $h(x)$ is maximized at $x^*:=c^{-1}\left(1-\sqrt{v/(2 \varepsilon c + v)} \right)<1/c$. To this end we note that 
\begin{align*}
h_\varepsilon'(x)= \varepsilon - \frac{vx(2-cx)}{2(1-cx)^2}.
\end{align*}  
Moreover, the denominator at $x^*$ is
\begin{align*}
2 (1-c x^*)^2&= \frac{2  v}{2 \varepsilon c + v} ,
\end{align*}
while the enumerator is
\begin{align*}
v x^*(2-c x^*)&= \frac{v}{c} \left( 1- \sqrt{\frac{v}{2 \varepsilon c + v}} \, \right) \left( 1+\sqrt{\frac{v}{2 \varepsilon c + v}} \, \right) = \frac{2 \varepsilon v}{2 \varepsilon c + v}.
\end{align*}
Inserting both equations in $h_{\varepsilon}'$ gives $h_\varepsilon'(x^*)= \varepsilon- \varepsilon=0$.
To verify that $h_\varepsilon''(x^*)<0$, we note that the second derivative is given by 
\begin{align*}
h_\varepsilon''(x)=\frac{-v}{(1-cx)^3}.
\end{align*}
Inserting $x^*$ gives 
\begin{align*}
h''_\varepsilon(x^*)=-v \left(\frac{v}{2 \varepsilon c + v} \right)^{-3/2} <0.
\end{align*}
Thus $x^*$ maximizes $h_\varepsilon$.

In the next step we calculate $h_\varepsilon(x^*)$. To this end we set 
\begin{align*}
t:=\sqrt{\frac{v}{2\varepsilon c+v}} 
\end{align*}
and note that $x^*=\frac{1-t}{c}$ and $1-cx^*=t$ holds true. 
We thus obtain 
\begin{align*}
h_\varepsilon(x^*)= \varepsilon \frac{1-t}{c}- \frac{v}{2} \frac{(1-t)^2}{tc^2} = \frac{2c\varepsilon(1-t)t}{2c^2t}- \frac{v(1-t)^2}{2c^2 t}=\frac{(1-t)(2 \varepsilon c t - v(1-t))}{2c^2 t}.
\end{align*}
Since we have $(2\varepsilon c + v)t=v/t$ we conclude
\begin{align*}
2 \varepsilon c t - v(1-t)=(2\varepsilon c + v)t -v = v \frac{1-t}{t}. 
\end{align*}
It follows that 
\begin{align*}
h_\varepsilon(x^*)=\frac{(1-t)(v(1-t)/t)}{2 c^2 t} = \frac{v (1-t)^2}{2 c^2 t^2}.
\end{align*}
For $u:=1/t=\sqrt{1+\frac{2\varepsilon c}{v}}$ we now have 
\begin{align*}
\frac{(1-t)^2}{t^2}= \left( \frac{1}{t}-1 \right)^2=(u-1)^2,
\end{align*}
and with $u^2+1=\frac{2v+2 \varepsilon c}{v}$
\begin{align*}
h_\varepsilon(x^*)=\frac{v}{2c^2}(u-1)^2= \frac{v(u^2-2u+1)}{2c^2}= \frac{2(\varepsilon c +v)-2vu}{2c^2}= \frac{\varepsilon c + v - v u }{c^2}.
\end{align*}
Multiplying numerator and denominator by $\varepsilon c + v + vu$ gives 
\begin{align*}
h_\varepsilon(x^*)=\frac{(\varepsilon c+v-v u)(\varepsilon c + v + v u)}{c^2(\varepsilon c+v+vu)}=\frac{(\varepsilon c + v)^2-(vu)^2}{c^2(\varepsilon c + v+vu)}.
\end{align*}
Moreover, we have 
\begin{align*}
(\varepsilon c+ v)^2-(v u)^2= \varepsilon^2c^2+2 \varepsilon c v+v^2-v^2\left(1+\frac{2\varepsilon c }{v} \right)= \varepsilon^2 c^2
\end{align*}
and hence 
\begin{align*}
h_\varepsilon(x^*)=\frac{\varepsilon^2 c^2}{c^2(\varepsilon c + v+vu)}= \frac{\varepsilon^2}{\varepsilon c + v + v \sqrt{1+\frac{2 \varepsilon c}{v}}}.
\end{align*}
We now show that for for $\varepsilon:= c \tau + \sqrt{2 v \tau}$ we have $h_{c\tau+\sqrt{2 v \tau}}(x^*)=\tau$. To this end we note that
\begin{align*}
\varepsilon^2=\tau( \tau c^2 +2 c \sqrt{2 v \tau} +2 v) \quad  \textup{and} \quad c \varepsilon+v = \tau c^2  + c \sqrt{2 v \tau} +v.
\end{align*}
This gives
\begin{align*}
v\sqrt{1+\frac{2 \varepsilon c}{v}}= \sqrt{v^2+2 \varepsilon c v} = \sqrt{ c^2 v \tau+2 c v \sqrt{2 v \tau}+v^2} = \sqrt{(v+c \sqrt{2 v \tau})^2} = v+c\sqrt{2v \tau},
\end{align*} 
and thus we conclude 
\begin{align*}
h_{c\tau+\sqrt{2 v \tau}}(x^*)=\frac{\tau ( \tau c^2 + 2 c \sqrt{2 v \tau}+2v)}{\tau c^2 + c \sqrt{2 v \tau} + v+v +c \sqrt{2 v \tau}} = \frac{\tau ( \tau c^2 + 2 c \sqrt{2 v \tau}+2v)}{ \tau c^2 + 2 c \sqrt{2 v \tau}+2v}= \tau.
\end{align*}
In summary we obtain 
\begin{align*}
\mu(Z \geq c \tau + \sqrt{2 v \tau} ) \leq \mathrm{e}^{-\tau}.
\end{align*}
\end{proof} 
\end{appendices}

\end{document}